\documentclass[a4paper, twoside, 12pt]{article}
\usepackage{amsmath,amsthm,amsfonts,latexsym,amscd,amssymb, enumerate}
\usepackage{hyperref}

\usepackage{float}

\numberwithin{equation}{section}
\newtheorem{theorem}{Theorem}[section]
\newtheorem{lemma}{Lemma}[section]
\newtheorem{corollary}{Corollary}[section]
\newtheorem{proposition}{Proposition}[section]
\newtheorem{definition}{Definition}[section]

\theoremstyle{remark}

\date{}
\title{\textbf{On some optimal inequalities for bi-slant submanifolds in metallic Riemannian space forms}}
\author{Harmandeep Kaur and Gauree Shanker\thanks{corresponding author, Email: gauree.shanker@cup.edu.in}}

\begin{document}
	\maketitle
	\begin{center}
		\textbf{Abstract}
	\end{center}
	In this paper, we derive some important optimal relationships for bi-slant submanifolds in metallic Riemannian product space forms enriching the understanding of their geometric properties and deepening the connection between intrinsic and extrinsic curvature invariants. We establish generalized Wintgen inequality for bi-slant submanifolds in metallic Riemannian product space forms and discussed the equality case. Next we derive optimal inequalities involving $\delta$-invariants, also known as Chen-invariants and discuss the conditions for Chen ideal submanifolds. Further, we derive optimal relationships involving Ricci curvature and shape operator invariants along with the discussion about the equality cases. In the last section, we establish optimal inequalities involving generalized normalized $\delta$-Casorati curvatures for bi-slant submanifolds of metallic Riemannian product space form and discuss the conditions under which the equality holds. Furthermore, we examine how the main findings specialize to slant, semi-slant, hemi-slant, and semi-invariant submanifolds in metallic Riemannian product space forms, offering a better understanding of their geometric characteristics.

	\section{Introduction}
	
	A fundamental problem in submanifold theory is the immersibility of a Riemannian manifold into a Euclidean space (or, more generally, a space form). Nash's imbedding theorem ensures that every Riemannian manifold can be isometrically immersed in a sufficiently higher dimensional Euclidean space \cite{gw112}. It inspired researchers to explore the vast area of submanifold geometry. While Nash’s theorem provides a way to utilize extrinsic geometry for studying intrinsic Riemannian geometry, the challenge arises from the high codimension of such immersions, making them difficult to analyze. Addressing this requires establishing optimal relationships between intrinsic and extrinsic invariants of submanifolds. Several well-known results in differential geometry, including the isoperimetric inequality, Chern-Lashof inequality, and Gauss-Bonnet theorem can be viewed as foundational examples of optimal relationships between intrinsic and extrinsic geometric invariants, aligning closely with the focus of this study. With this motivation, Chen introduced a new intrinsic curvature invariant, known as $\delta$-curvature invariant and established optimal relationship between this invariant and the main extrinsic invariants for arbitrary Riemannian submanifolds \cite{chen1}. Working in this direction, several optimal relationships involving $\delta$-curvature invariants have been derived for various classes of submanifolds in different ambient manifolds \cite{gwchenbook}.  \\ 
The Wintgen inequality represents another significant optimal relationship between intrinsic and extrinsic curvature invariants. It is a sharp geometric relationship for surfaces in 4-dimensional Euclidean space involving Gauss curvature (intrinsic invariant), normal curvature, and squared mean curvature (extrinsic invariants), established by P. Wintgen \cite{gw11}. Guadalupe, et al. \cite{gw12} extended this study of  surfaces in $4$-dimensional Euclidean space to surfaces of arbitrary codimension in real space forms. In continuation of this, De Smet, et al. \cite{gw4} studied the Wintgen inequality and proposed a conjecture for general Riemannian submanifolds in real space forms, thereby giving the generalized Wintgen inequality, also known as DDVV conjecture. Ge, et al. \cite{gw8} and Lu \cite{gw13} independently proved the DDVV conjecture. Since then, the study of Wintgen inequality and Wintgen ideal submanifolds has been a tremendous and fascinating area of research \cite{gw119}. \\  
A deep understanding of submanifold geometry hinges on both intrinsic and extrinsic curvature invariants. Casorati curvature, introduced in \cite{casoext1}, is an important extrinsic invariant. It was preferred by Casorati over the traditional Gauss curvature as it aligns more closely with the intuitive understanding of curvature. In computer vision, it represents the bending energy of surfaces in $\mathbb{R}^3$ \cite{casoext2}. This concept has been extended to arbitrary submanifolds in Riemannian geometry as the normalized Casorati curvature and generalized normalized $\delta$-Casorati curvature \cite{h11}. Nowadays, Casorati curvature has become a key tool in deriving optimal inequalities for submanifolds in various ambient spaces \cite{gwcasre}. \\
On the other hand, polynomial structures on smooth manifolds, defined by Goldberg and Yano \cite{gw111}, provide a natural extension of classical differential geometry by incorporating algebraic properties into the geometric framework. These polynomial structures play an important role in differential geometry, particularly in the study of submanifolds, curvature properties, and geometric flows. They also find applications in mathematical physics, providing models for generalized geometric spaces in relativity and string theory. A smooth tensor field $f$ of type $(1,1)$ on a smooth manifold $M$ is said to define a polynomial structure of degree $d$ on $M$ if $d$ is the smallest integer for which the powers $I$, $f$, ..., $f^d$ are dependent, and $f$ has constant rank on $M$  \cite{gw111}. The almost complex structure and the almost contact structure are specific examples of polynomial structures defined on the manifold $M.$ Derived from the golden ratio, the golden structure is a fundamental polynomial structure of degree $2$ on smooth manifolds, characterized by a $(1,1)$-tensor field $\psi$ satisfying $\psi^2 = \psi + I$. Inspired by the golden mean, De Spinadel \cite{gw10} introduced a more general family of polynomial structures, known as the metallic means family, which led to the development of an important class of Riemannian manifolds called metallic Riemannian manifolds \cite{gw1}. Metallic Riemannian manifolds represent a significant advancement in the field of differential geometry, merging classical Riemannian concepts with innovative structures inspired by metallic means. The incorporation of metallic structures enriches the study of geometric properties, introducing new perspectives beyond those of traditional Riemannian manifolds \cite{gw125}. \\
Motivated by the above literature, the main focus of this paper is to contribute further to this line of research by establishing new optimal relationships between intrinsic and extrinsic curvature invariants in the context of bi-slant submanifolds of metallic Riemannian space forms. This paper is organized as follows: Section \ref{sec1w} is devoted to the preliminary notations and results required to prove the main objective of our study. Section \ref{sec2w} comprises the main research findings. In section \ref{sec2.1w}, we derive the generalized Wintgen inequality for bi-slant submanifolds of metallic Riemannian space forms and discussed the equality case. In section \ref{sec2.2w}, we establish optimal relationships involving $\delta$-invariants and discuss the conditions for Chen ideal submanifolds. In section \ref{sec2.3w}, we derive optimal inequalities involving Ricci curvature and shape operator along with discussion about the equality case. Further, in section \ref{sec2.4w}, we establish optimal inequalities involving generalized normalized $\delta$-Casorati curvatures for bi-slant submanifolds of metallic Riemannian product space form and discussed the conditions under which the equality holds. Finally, we discuss the related results for the special cases for semi-slant, hemi-slant, semi-invariant, and slant submanifolds of metallic Riemannian product space form.

	\section{Preliminaries}\label{sec1w}
	In this section, we recall the fundamental concepts required for our study. \\
	\begin{definition}\cite{gw3}
		A $(1, 1)$ tensor field $\mathcal{F}$ defined on a Riemannian manifold $M$ is called almost product structure if for any vector field $X \in \Gamma{(TM)}$, the following equality holds
		\begin{equation*}
			\mathcal{F}^2 = I.	
		\end{equation*}
	\end{definition}
	\begin{definition}\cite{gw3}
		A Riemannian manifold $(M, g, \mathcal{F})$ is called an almost product Riemannian manifold if the Riemannian metric $g$ is $\mathcal{F}$ compatible i.e.,
		\begin{equation*}
			g(X, \mathcal{F} Y) = g(\mathcal{F} X, Y)
		\end{equation*}
		for any vector fields $X, Y \in \Gamma{(TM)}$.
	\end{definition}
	\begin{definition}\cite{gw1}
		A polynomial structure $\phi$ on a smooth manifold $M$ is called a metallic structure if it is determined by a $(1, 1)$ tensor field $\phi$ which satisfies the equation
		\begin{equation*}
			\phi^2 = p \phi + q I,
		\end{equation*}
		where $p$, $q$ are positive integers and $I$ is the identity endomorphism on $\Gamma{(TM)}$.
	\end{definition}
	
	\begin{definition}\cite{gw1}
		A Riemannian manifold $(M, g)$ endowed with a metallic structure $\phi$ is called a metallic Riemannian manifold if the Riemannian metric $g$ is $\phi$ compatible i.e., for any vector fields $X, Y \in \Gamma{(TM)}$, the following relation holds
		\begin{equation*}
			g(X, \phi Y) = g(\phi X, Y).
		\end{equation*}   
	\end{definition}
	Replacing $X$ with $\phi X$, we get
	\begin{equation}\label{w2}
		g(\phi X, \phi Y) = p g(X, \phi Y) + q g(X, Y).
	\end{equation}   
	In particular, when $p = q = 1$, the metallic structure becomes the golden structure and metallic Riemannian manifold is the golden Riemannian manifold.\\
The metallic structure and the almost product structure are closely related. The relationship between the almost product structure and the metallic structure is given by the following proposition.    
	\begin{proposition}\cite{gw1}\label{p1}
		Every almost product structure $\mathcal{F}$ induces two metallic structures on a Riemannian manifold $M$ which are given as follows:
		\begin{equation*}\label{w5}
			\phi_1 = \frac{p}{2}I + \bigg(\frac{2 \sigma_{p, q} - p}{2}\bigg) \mathcal{F},\  \phi_2 = \frac{p}{2}I - \bigg(\frac{2 \sigma_{p, q} - p}{2}\bigg) \mathcal{F}.
		\end{equation*}
		Conversely, every metallic structure $\phi$ on $M$ induces two almost product structures on the manifold, given by
		\begin{equation}\label{w6}
			\mathcal{F} = \pm \bigg(\frac{2}{2\sigma_{p, q} - p}\phi - \frac{p}{2\sigma_{p, q} - p}I\bigg),
		\end{equation}
		where $\sigma_{p, q} = \dfrac{p + \sqrt{p^2 + 4q}}{2}$ represents the members of metallic means family. In particular, if the almost product structure $\mathcal{F}$ is a Riemannian structure, then $\phi_1$ and $\phi_2$ are also metallic Riemannian structures.
	\end{proposition}
For simplicity, let us denote $\sqrt{p^2 + 4q}$ by $\alpha$ i.e.,
	\begin{equation}\label{w3}
		\alpha = \sqrt{p^2 + 4q} = 2\sigma_{p, q} - p.
	\end{equation}
	\begin{definition}\cite{gw2}
		A linear connection $\nabla$ on a metallic Riemannian manifold $(M, g, \phi)$ is said to be a $\phi$-connection if $\phi$ is covariantly constant with respect to $\nabla$ i.e., $\nabla \phi = 0$. If the Levi-Civita connection of $g$ is a $\phi$-connection, then $(M, g, \phi)$ is called a locally metallic Riemannian manifold.
	\end{definition}
	Let $M = (M_1(c_1) \times M_2(c_2), g, \mathcal{F})$ be a locally product Riemannian manifold. The Riemann curvature tensor $\tilde{R}$ of $M$ is given by \cite{gw3}
	\begin{equation}\label{w7}
		\begin{split}
			\tilde{R}(X, Y)Z =& \frac{1}{4}(c_1 + c_2)[g(Y, Z)X - g(X, Z)Y + g(\mathcal{F}Y, Z)\mathcal{F}X - g(\mathcal{F}X, Z)\mathcal{F}Y] +\\& \frac{1}{4}(c_1 - c_2)[g(\mathcal{F}Y, Z)X - g(\mathcal{F}X, Z)Y + g(Y, Z)\mathcal{F}X - g(X, Z)\mathcal{F}Y].
		\end{split}
	\end{equation}
	Using \eqref{w6}, \eqref{w3} and \eqref{w7}, we obtain the Riemann curvature tensor of the locally metallic product Riemannian space form $M$ as
\begin{equation}\label{w8}
	\begin{split}
		\tilde{R}(X, Y)Z =& \frac{(c_1 + c_2)}{4\alpha^2}\big\{\alpha^2[g(Y, Z)X - g(X, Z)Y] + 4[g(\phi Y, Z)\phi X - g(\phi X, Z)\phi Y] +\\& p^2[g(Y, Z)X - g(X, Z)Y] + 2p[g(\phi X, Z)Y + g(X, Z)\phi Y - g(\phi Y, Z)X \\&- g(Y, Z)\phi X]\big\} \pm \frac{(c_1 - c_2)}{2\alpha}\big\{g(\phi Y, Z)X + g(Y, Z)\phi X - g(\phi X, Z)Y - \\&g(X, Z)\phi Y + p[g(X, Z)Y - g(Y, Z)X]\big\}.	
	\end{split}
\end{equation}
		Let $N$ be an $n$-dimensional submanifold of a Riemannian manifold $(M, \tilde{g})$ of dimension $m$. Let us denote the induced metric on $N$ by the same symbol $g$. Let $R$ and $\tilde{R}$ represent the Riemann curvature tensor of $N$ and $M$, respectively. Then, for the vector fields $X, Y, Z, W \in \Gamma(TN)$, the Gauss equation is given by
	\begin{equation}\label{w9}
		R(X, Y, Z, W) = \tilde{R}(X, Y, Z, W) + g(h(X, W), h(Y, Z)) - g(h(X, Z), h(Y, W)),
	\end{equation}
	where $h$ denotes the second fundamental form of $N$ in $M$.\\ 
	Let $\{e_1, ..., e_n\}$ and $\{e_{n+1}, ..., e_m\}$ be the local orthonormal bases of the tangent space $T_xN$ and the normal space $(T_xN)^\perp$ of $N$ in $M$ at $x \in N$. Then the mean curvature vector $\mathcal{H}$ of $N$ is given as
	\begin{equation*}\label{w10}
		\mathcal{H} = \sum_{i = 1}^{n}\frac{1}{n}h(e_i, e_i). 
	\end{equation*}
	The squared norms of the mean curvature vector $\mathcal{H}$ and second fundamental form $h$ are given as
	\begin{equation}\label{w16}
		\parallel\mathcal{H}\parallel^2 = \frac{1}{n^2} \sum_{r = n+1}^{m}\big(\sum_{i = 1}^{n} h_{ii}^r\big)^2, ~~\parallel h \parallel^2 = \sum_{r = n+1}^{m}\sum_{i, j = 1}^{n} (h_{ij}^r)^2.
	\end{equation}
The Casorati curvature $\mathcal{C}$ is given as
	\begin{equation}\label{w126}
	\mathcal{C} = \frac{1}{n}\parallel h \parallel^2.
	\end{equation}
	
	\begin{definition}\cite{h11}
The generalized normalized $\delta$-Casorati curvatures	$\delta_\mathcal{C}(u; n-1)$ and $\hat{\delta}_\mathcal{C}(u; n-1)$ of the submanifold $N$ of $M$ are defined for $u \neq n(n-1), u>0, u \in \mathbb{R}$ as
\begin{equation}
	[\delta_\mathcal{C}(u; n-1)]_x = u \mathcal{C}_x + a(u).\inf\{\mathcal{C}(W) | W \ a \ hyperplane \ of \ T_xN\},
\end{equation}
if $0<u<n(n-1)$, and
\begin{equation}
[ \hat{\delta}_\mathcal{C}(u; n-1)]_x = u \mathcal{C}_x + a(u).\sup\{\mathcal{C}(W) | W \ a \ hyperplane \ of \ T_xN\},	
\end{equation}
if $u > n(n-1)$, where $a(u)$ is given by
\begin{equation*}
	a(u) = \dfrac{(n - 1)(u + n)(n^2 - n - u)}{n u}.
\end{equation*} 
\end{definition}

\begin{definition}\cite{gwquasiu}
 A submanifold $N$ of $M$ is invariantly quasi-umbilical if there exist $(m-n)$ orthogonal unit normal vectors $\{e_{n+1}, ..., e_m\}$ such that the shape
operators with respect to all directions $e_r$ have an eigenvalue of multiplicity $(n-1)$ and that for each $e_r$ the distinguished eigen direction is the same.
\end{definition}	
The scalar curvature $\tau$ and normalized scalar curvature $\rho$ are given as follows:
	\begin{equation}\label{w12}
		\tau = \sum_{1 \leq i < j \leq n}^{} K(e_i \wedge e_j) = \sum_{1 \leq i < j \leq n}^{} R(e_i, e_j, e_j, e_i), ~~~\rho = \frac{2 \tau}{n(n-1)}, 
	\end{equation}
where $K(e_i \wedge e_j)$ represents the sectional curvature of the 2-plane section spanned by $e_i$ and $e_j$.\\
	The normalized scalar curvature of the normal bundle $(TN)^{\perp}$, denoted by $\rho^{\perp}$, is given as
	\begin{equation}\label{w13}
		\rho^{\perp} = \frac{2}{n(n-1)}\bigg[\sum_{1 \leq i < j \leq n}^{} \sum_{n+1 \leq t < s \leq m}^{}g\big(R^{\perp}(e_i, e_j)e_t, e_s\big)^2\bigg]^{\frac{1}{2}},
	\end{equation}
	where $R^{\perp}$ is the Riemann curvature tensor of the normal bundle $(TN)^{\perp}$.
\begin{definition}\cite{h9}
Let $(N,g)$ be a Riemannian manifold. For each given $k$-tuple $(n_1, ..., n_k) \in S(n)$, the finite set consisting of all $k$-tuples $(n_1, ..., n_k)$ of integers $2 \leq n_1, ..., n_k \leq n-1$ with $n_1 + ... + n_k \leq n$, the $\delta$-invariants are defined by
 $$\delta(n_1, ..., n_k) = \tau - \inf\{\tau(\mathcal{L}_1) + ... + \tau(\mathcal{L}_k))\},$$
 $$\hat{\delta}(n_1, ..., n_k) = \tau - \sup\{\tau(\mathcal{L}_1) + ... + \tau(\mathcal{L}_k))\},$$
 where $\mathcal{L}_1, ..., \mathcal{L}_k$ run over all $k$ mutually orthogonal subspaces of $T_xN, x\in N$ such that dim$\mathcal{L}_j = n_j$, $j = 1, ..., k$.
\end{definition}

\begin{definition}\label{omegadef}\cite{book}
Define a Riemannian invariant $\Omega_k$ on a Riemannian manifold $N$ by
\begin{equation}\label{omega11}
\Omega_k(x) = \frac{1}{(k-1)}\inf_{\mathcal{L}^k, X} Ric_{\mathcal{L}^k}(X), x \in N,
\end{equation}	
where $\mathcal{L}^k$ runs over all non-degenerate $k$-plane sections at $x$ and $X$ runs over all unit vectors in $\mathcal{L}^k$.
\end{definition}	
Let $N$ be an $n$-dimensional Riemannian submanifold of a real space form $M(c)$ of constant sectional curvature $c$ and dimension $m$. Then the normal scalar curvature conjecture, also known as the DDVV conjecture is given as\cite{gw4}
	\begin{equation}\label{w14}
		\rho + \rho^{\perp} \leq \parallel\mathcal{H}\parallel^2 + c	
	\end{equation}
	In terms of components of the second fundamental form, \eqref{w14} can be reformulated as\cite{gw5}
	\begin{equation}\label{w15}\begin{split}
\sum_{r = n+1}^{m}\sum_{1 \leq i < j \leq n}^{}\big(h_{ij}^r\big)^2 \geq &\bigg\{\sum_{n+1 \leq t < s \leq m-n}^{}\sum_{1 \leq i < j \leq n}\bigg[\sum_{k=1}^n\big(h_{ik}^th_{jk}^{s} - h_{ik}^sh_{jk}^t\big)\bigg]\bigg\}^{\frac{1}{2}} - \\& \frac{1}{2n}\sum_{r = n+1}^{m}\sum_{1 \leq i < j \leq n}\big(h_{ii}^r - h_{jj}^r\big)^2.
		\end{split}
	\end{equation}
Let $N$ be an $n$-dimensional submanifold isometrically immersed in metallic Riemannian manifold $(M, g, \phi)$ of dimension $m$. Let $\theta$ be the angle between $\phi X$ and $T_xN$ at $x \in N$. The tangent space of $T_xM$ of $M$ has a direct sum decomposition in terms of tangent and normal space of $N$ at $x \in N$, given by
	\begin{equation*}
		T_xM = T_xN + (T_xN)^{\perp}
	\end{equation*}
	Consider the decomposition of $\phi X$ and $\phi U$, for any $X \in T_xN$ and $U \in (T_xN)^{\perp}$, into tangential and normal components as
	\begin{equation}\label{w21}
		\phi X = TX + NX,\ \  \phi U = tU + nU,	
	\end{equation} 
	where $TX$, $tU$ are the tangential components and $NX$, $nU$ are the normal components of $\phi X$, $\phi U$, respectively. 
	\begin{definition}\cite{gw2}
		The submanifold $N$ is a slant submanifold if for any $x \in N$ and $X \in T_xN$, the angle $\theta(X)$ between $\phi X$ and $T_xN$ is constant, the angle $\theta(X)$ is known as the slant angle.
	\end{definition}
In particular, if $\theta = 0$, $N$ is an invariant submanifold and if $\theta = \frac{\pi}{2}$, $N$ is an anti-invariant submanifold; otherwise $N$ is a proper slant submanifold. The slant angle $\theta$ is given by 
	\begin{equation}\label{wtheta}
		\cos\theta = \frac{g(\phi X, TX)}{\parallel \phi X \parallel \parallel TX \parallel} = \frac{\parallel TX \parallel}{\parallel \phi X \parallel}
	\end{equation} 
	\begin{definition}\cite{gw6}
		A differentiable distribution $\mathcal{D}$ on $N$ is called a slant distribution if the angle $\theta_\mathcal{D}$ between $\phi X$ and the vector subspace $\mathcal{D}_x$ is constant, for any $x \in N$ and any non zero vector field $X \in \Gamma(\mathcal{D}_x)$. The constant angle $\theta_\mathcal{D}$ is called the slant angle of the distribution $\mathcal{D}$.
	\end{definition}
	\begin{lemma}\cite{gw6}\label{lemma2}
		Let $\mathcal{D}$ be a differentiable distribution on a submanifold $N$ of a metallic Riemannian manifold $(M, g, \phi)$. The distribution $\mathcal{D}$ is a slant distribution if and only if there exists a constant $\lambda \in [0, 1]$ such that
		\begin{equation*}\label{w19}
			(P_{\mathcal{D}}T)^2X = \lambda\ (p~ P_{\mathcal{D}}TX + qX)
		\end{equation*}
		for any $X \in \Gamma(D)$, where $P_{\mathcal{D}}$ is the orthogonal projection on $\mathcal{D}$. Moreover, if $\theta_\mathcal{D}$ is the slant angle of $D$, then it satisfies $\lambda = \cos^2\theta_\mathcal{D}$.
	\end{lemma}
	\begin{definition}\cite{gw6}
		A submanifold $N$ in a metallic Riemannian manifold $(M, g, \phi)$ is called bi-slant if there exist two orthogonal differentiable distributions $\mathcal{D}_1$ and $\mathcal{D}_2$ on $N$ such that $TN = \mathcal{D}_1 \oplus \mathcal{D}_2$ and $\mathcal{D}_1$ and $\mathcal{D}_2$ are slant distributions with slant angles $\theta_1$ and $\theta_2$, respectively.

	\end{definition} 
	Let $P_1$ and $P_2$ denote the orthogonal projections on $\mathcal{D}_1$ and $\mathcal{D}_2$, respectively. Hence, for any $X \in \Gamma(TN)$, we have $X = P_1X + P_2X,$ where $P_1X \in \Gamma(\mathcal{D}_1),$ and $P_2X \in \Gamma(\mathcal{D}_2)$. Also, using \eqref{w21}, we have 
	$$\phi X = TP_1X + TP_2X + NP_1X + NP_2X.$$
	From equations \eqref{wtheta} and \eqref{w2}, in view of Lemma \ref{lemma2},  we obtain
	\begin{equation}\label{theta1}
	g(TP_iX, TP_iX) = \cos^2\theta_i [pg(\phi P_iX, P_iX) + qg(P_iX, P_iX)], i= 1, 2.
	\end{equation}

	\begin{lemma}\cite{deltaa}\label{lemmadelta}
	Let $a_1, ..., a_n, \epsilon$ be $n+1 (n \geq 2)$ real numbers such that 
	\begin{equation}\label{lemmadel1}
	\big(\sum_{i=1}^{n}a_i\big)^2 = (n-1)\big(\epsilon + \sum_{i=1}^{n}a_{i}^2\big).
	\end{equation}
	Then $2a_1a_2 \geq \epsilon,$ equality holds if and only if $a_1 + a_2 = a_3 = ... = a_n.$
	\end{lemma}

	\begin{table}[H] 
		\begin{center}
			\caption{Classification of submanifolds}
			\label{Table 1}
			\begin{tabular}{l ll ll}
				\hline
				Type of submanifold & Type of $\mathcal{D}_1$  & Type of $\mathcal{D}_2$ & Angle $\theta_1$ & Angle $\theta_2$ \\
				\hline
				Bi-slant  & slant & slant & $\theta_1 \in [0, \frac{\pi}{2}]$ & $\theta_2 \in [0, \frac{\pi}{2}]$ \\
				Semi-slant  & invariant & slant & $\theta_1 = 0$ & $\theta_2 = \theta \in [0, \frac{\pi}{2}]$ \\
				Hemi-slant  & slant & anti-invariant & $\theta_1 = \theta \in [0, \frac{\pi}{2}]$ & $\theta_2 = \frac{\pi}{2}$ \\
				Semi-invariant  & invariant & anti-invariant & $\theta_1 = 0$ & $\theta_2 = \frac{\pi}{2}$ \\
				Slant  & \multicolumn{2}{l}{either $\mathcal{D}_1 = 0$ or $\mathcal{D}_2 = 0$} & \multicolumn{2}{l}{$\theta_1 = \theta_2 = \theta \in [0, \frac{\pi}{2}]$}  \\
			\hline
			\end{tabular}
		\end{center}
	\end{table}
	
	\section{Main Results}\label{sec2w}
	
	\subsection{Generalized Wintgen inequality}\label{sec2.1w}
	In this section, we derive the generalized Wintgen inequality for bi-slant submanifolds in metallic Riemannian space forms.
	
	\begin{theorem}\label{th1}
		Let $N$ be an $n$-dimensional bi-slant submanifold of a locally metallic product space form $M = (M_1(c_1) \times M_2(c_2), g, \phi$) of dimension $m$. Then the following optimal relationship holds:
		\begin{equation*}\label{w1}
			\begin{split}
				\parallel\mathcal{H}\parallel^2 \geq & \rho + \rho^\perp - \frac{(c_1+c_2)}{2\alpha^2}(p^2+2q) + \frac{(c_1+c_2)}{n\alpha^2(n-1)}\big[(n-1)(p~ tr\phi - n~ tr^2\phi) + \\ & \cos^2\theta_1(p~ tr TP_1 +d_1q) + \cos^2\theta_2(p~ tr TP_2 + d_2q)\big] + \frac{(c_1-c_2)}{2n\alpha}(np - 2~ tr\phi)
			\end{split}
		\end{equation*}
		and the equality holds at a point $x \in N$ if and only if the shape operators $\mathcal{A}_r, r = n+1, ..., m;$ take the following forms with respect to some suitable orthonormal basis $\{e_1, ..., e_n\}$ of $T_xN$ and $\{e_{n+1}, ..., e_m\}$ of $(T_xN)^{\perp}$,
		\begin{equation}\label{w26}
			\mathcal{A}_{n+1} =
			\begin{pmatrix} 
				\alpha_1 & \beta & 0 & \cdots & 0 \\
				\beta & \alpha_1 & 0 & \cdots & 0 \\
				0 & 0 & \alpha_1 & \cdots & 0 \\
				\vdots & \vdots & \vdots & \ddots & \vdots\\
				0 & 0 & 0 & \cdots & \alpha_1 \\	
			\end{pmatrix},
				\quad
			\mathcal{A}_{n+2} =
			\begin{pmatrix} 
				\alpha_2 + \beta & 0 & 0 & \cdots & 0 \\
				0 & \alpha_2 - \beta & 0 & \cdots & 0 \\
				0 & 0 & \alpha_2 & \cdots & 0 \\ 
				\vdots & \vdots & \vdots & \ddots & \vdots\\
				0 & 0 & 0 & \cdots & \alpha_2 \\	
			\end{pmatrix}
		\end{equation}
		\begin{equation}\label{w28}
			\mathcal{A}_{n+3} =
			\begin{pmatrix} 
				\alpha_3 & 0 & 0 & \cdots & 0 \\
				0 & \alpha_3 & 0 & \cdots & 0 \\
				0 & 0 & \alpha_3 & \cdots & 0 \\
				\vdots & \vdots & \vdots & \ddots & \vdots\\
				0 & 0 & 0 & \cdots & \alpha_3 \\	
			\end{pmatrix},
			\quad
		\mathcal{A}_{n+4} = ... = \mathcal{A}_m = 0,
		\end{equation} 
		where $\alpha_1$, $\alpha_2$, $\alpha_3$ and $\beta$ are real functions on $N$.
		\end{theorem}
		\begin{proof}
Consider the local orthonormal frame $\{e_1, ..., e_n\}$ of the tangent space $T_xN$ and $\{e_{n+1}, ..., e_m\}$ of the normal space $(T_xN)^{\perp}$ of $N$ at $x \in N$. Using \eqref{w9} and \eqref{w8}, we obtain the scalar curvature of the submanifold $N$ of $M$ as
\begin{equation}\label{wt3}
\begin{split}
2 \tau =& \frac{n(n-1)(c_1+c_2)}{4\alpha^2}(2p^2+4q) + \frac{(c_1+c_2)}{\alpha^2}\bigg[\sum_{1 \leq i \neq j \leq n}^{}g(\phi e_i, e_i)g(\phi e_j, e_j) - \\ & \parallel T \parallel^2 - p(n-1)\sum_{i = 1}^{n}g(\phi e_i, e_i)\bigg] + \frac{(c_1-c_2)(n-1)}{2\alpha}\bigg[2\sum_{i = 1}^{n}g(\phi e_i, e_i) - np\bigg] \\ & + \sum_{r=n+1}^{m}\sum_{1 \leq i \neq j \leq n}^{}[h_{ii}^rh_{jj}^r - (h_{ij}^r)^2].	
\end{split}
\end{equation}
Using \eqref{theta1} and \eqref{wt3}, we get
\begin{equation}\label{wt4}
\begin{split}
2 \tau &= \frac{n(n-1)(c_1+c_2)}{2\alpha^2}(p^2+2q) + \frac{(c_1+c_2)}{\alpha^2}\big[(n-1)(n~ tr^2\phi - p~ tr\phi) - \cos^2\theta_1 \\ &(p~ tr TP_1 +d_1q) - \cos^2\theta_2(p~ tr TP_2 + d_2q)\big] + \frac{(n-1)(c_1-c_2)}{2\alpha}(2~ tr\phi - np) + \\ & 2\sum_{r = n+1}^{m}\sum_{1 \leq i < j \leq n}^{}[h_{ii}^rh_{jj}^r - (h_{ij}^r)^2].	
\end{split}
\end{equation}
			From \eqref{w16}, we obtain
	\begin{equation}\label{wt5}
		\begin{split}
			2\sum_{r = n+1}^{m}\sum_{1 \leq i < j \leq n}h_{ii}^r h_{jj}^r = n(n-1)\parallel\mathcal{H}\parallel^2 - \frac{1}{n}\sum_{r = n+1}^{m}\sum_{1 \leq i < j \leq n}(h_{ii}^r - h_{jj}^r)^2.  
		\end{split}
	\end{equation}
Making use of \eqref{w13} and \eqref{w15}, we get
\begin{equation}\label{wt6}
	\begin{split}
	-2 \sum_{r = n+1}^{m}\sum_{1 \leq i < j \leq n}(h_{ij}^r)^2 \leq \frac{1}{n}\sum_{r = n+1}^{m}\sum_{1 \leq i < j \leq n}(h_{ii}^r - h_{jj}^r)^2 - n(n-1)\rho^\perp.	
	\end{split}
\end{equation}
From \eqref{wt4}, \eqref{wt5} and \eqref{wt6}, we get
\begin{equation*}
	\begin{split}
2 \tau &\leq \frac{n(n-1)(c_1+c_2)}{2\alpha^2}(p^2+2q) + \frac{(c_1+c_2)}{\alpha^2}\big[(n-1)(n~ tr^2\phi - p~ tr\phi) - \\ & \cos^2\theta_1(p~ tr TP_1 +d_1q) - \cos^2\theta_2(p~ tr TP_2 + d_2q)\big] + \frac{(n-1)(c_1-c_2)}{2\alpha} \\ &(2~ tr\phi - np) + n(n-1)\parallel\mathcal{H}\parallel^2 - n(n-1)\rho^\perp. 		
	\end{split}
\end{equation*}
Hence, using \eqref{w12} in above equation, we obtain the desired inequality. For the equality case, proceeding with similar argumentation as in corollary $(1.2)$ of \cite{gw8}, we conclude that the equality holds at a point $x \in N$ if and only if with respect to some suitable orthonormal basis of the tangent space and the normal space of $N$ at the point $x$, the shape operators take the forms as given by \eqref{w26} and \eqref{w28}.
		\end{proof}
	
	\begin{corollary}\label{cor1}
		Let $N$ be an $n$-dimensional submanifold (as given in Table \ref{Table 1}) of a locally metallic product space form $M = (M_1(c_1) \times M_2(c_2), g, \phi$) of dimension $m$. Then, the following optimal relationships for the different types of slant submanifolds (given in Table \ref{Table 2}) hold and the equality holds if and only if for some suitable orthonormal basis  $\{e_1, ..., e_n\}$ of $T_xN$ and $\{e_{n+1}, ..., e_m\}$ of $(T_xN)^{\perp}$, the shape operator satisfies \eqref{w26} and \eqref{w28}.
\end{corollary}	

	\begin{table}[H]      
	\begin{center}
		\caption{Generalized Wintgen inequality for submanifolds of locally metallic product space form}
		\label{Table 2}
		\vspace{.06cm}
			\begin{tabular}{|l| p{10.5cm}|}
				\hline
				Type of $N$~~ & Generalized Wintgen Inequality\\
				\hline
				Semi-slant & $\parallel\mathcal{H}\parallel^2 \geq \rho + \rho^\perp - \dfrac{(c_1+c_2)}{2\alpha^2}(p^2+2q) + \dfrac{(c_1+c_2)}{n\alpha^2(n-1)}\big[(n-1)(p~ tr\phi - n~ tr^2\phi) + p~trTP_1 + d_1q + \cos^2\theta(p~trTP_2 + d_2q)\big] + \dfrac{(c_1-c_2)}{2\alpha n}(np -2~tr\phi)$\\
				\hline
				Hemi-slant & $\parallel\mathcal{H}\parallel^2 \geq \rho + \rho^\perp - \dfrac{(c_1+c_2)}{2\alpha^2}(p^2+2q) + \dfrac{(c_1+c_2)}{n\alpha^2(n-1)}\big[(n-1)(p~ tr\phi - n~ tr^2\phi) + \cos^2\theta(p~trTP_1 + d_1q)\big] + \dfrac{(c_1-c_2)}{2\alpha n}(np -2~tr\phi)$\\
				\hline
				Semi-invariant & $\parallel\mathcal{H}\parallel^2 \geq \rho + \rho^\perp - \dfrac{(c_1+c_2)}{2\alpha^2}(p^2+2q) + \dfrac{(c_1+c_2)}{n\alpha^2(n-1)}\big[(n-1)(p~ tr\phi - n~ tr^2\phi) + p~trTP_1 + d_1q \big]+ \dfrac{(c_1-c_2)}{2\alpha n}(np -2~tr\phi)$\\
				\hline
				Slant & $\parallel\mathcal{H}\parallel^2 \geq \rho + \rho^\perp - \dfrac{(c_1+c_2)}{2\alpha^2}(p^2+2q) + \dfrac{(c_1+c_2)}{n\alpha^2(n-1)}\big[(n-1)(p~ tr\phi - n~ tr^2\phi) + \cos^2\theta(p~trT + nq)\big] + \dfrac{(c_1-c_2)}{2\alpha n}(np -2~tr\phi)$\\
				\hline
			\end{tabular}	
		\end{center}
	\end{table}	
%


	\subsection{Optimal inequalities involving $\delta$-invariants}\label{sec2.2w}
In this section, we establish optimal inequalities involving $\delta$-invariants for submanifolds in metallic Riemannian space forms.	
	\begin{theorem}\label{wcth1}
	Let $N$ be an $n$-dimensional bi-slant submanifold of a locally metallic product space form $M = (M_1(c_1) \times M_2(c_2), g, \phi$) of dimension $m$. Then for each $k$-tuple $(n_1, ..., n_k) \in S(n),$ we have
	 \begin{align}\label{wc1}
\delta(n_1, ..., n_k) \leq&  \frac{(c_1+c_2)}{2\alpha^2}[b(n_1, ..., n_k)(p^2+2q+2~tr^2\phi) - d(n_1, ..., n_k)p~tr\phi +  \nonumber \\
&(k-1)( \cos^2\theta_1 (p~ tr TP_1 +d_1q) + \cos^2\theta_2(p~ tr TP_2 + d_2q))] +  \nonumber \\
& \frac{(c_1-c_2)}{4\alpha}[2~tr\phi ~d(n_1, ..., n_k) - p~b(n_1, ..., n_k)] + c(n_1, ..., n_k)\parallel \mathcal{H}\parallel^2,
\end{align}
where 
 \begin{align*}
&b(n_1, ..., n_k) = \frac{1}{2}[n(n-1) - \sum_{i=1}^{k}n_i(n_i-1)],  \nonumber \\
&c(n_1, ..., n_k) = \dfrac{(n+k-1-\sum_{i=1}^{k}n_i)}{2(n+k-\sum_{i=1}^{k}n_i)} n^2  \nonumber,  \\ 
&d(n_1, ..., n_k) = (n-1) - \sum_{i=1}^{k}(n_i-1)  \nonumber
\end{align*}
	and the equality holds at $x \in N$ if and only if for some suitable orthonormal basis $\{e_1, ..., e_n\}$ of $T_xN$ and $\{e_{n+1}, ..., e_m\}$ of $(T_xN)^{\perp}$, the shape operator takes the following form
	\begin{equation}\label{wce2}
			\mathcal{A}_{n+1} =
			\begin{pmatrix}  
				A_{\Delta_1} & \cdots & 0 & \\
				\vdots & \ddots & \vdots & 0 \\
				0 & \cdots & A_{\Delta_k} & \\
				   & 0          &                     & \nu I    \\
			\end{pmatrix},
			\quad
			\mathcal{A}_{r} =
			\begin{pmatrix}  
	       A_{\Delta_1}^r & \cdots & 0                   & \\
		         \vdots & \ddots & \vdots           & 0 \\
				0 & \cdots & A_{\Delta_k}^r & \\
				   & 0          &                     &  O   \\
						\end{pmatrix},
		\end{equation}
	where $I$ is an identity matrix and $O$ is the zero matrix of order $n_{k+1},$ $A_{\Delta_j}$ and $A_{\Delta_j}^r$ are symmetric submatrices of order $n_j$ such that $trace(A_{\Delta_j}) = \nu$ and $trace(A_{\Delta_j}^r) = 0,$ $1 \leq j \leq k,$ $r = \{n+2, ..., m\}.$
	\end{theorem}
	\begin{proof}
	Consider an orthonormal basis $\{e_1, ..., e_m\}$ of $T_xM$ at $x \in M.$ Let us choose $k$-mutually orthogonal subspaces $\mathcal{L}_1, ..., \mathcal{L}_k$ of $T_xN$ such that 
	$$\mathcal{L}_i = span\{e_{\beta_i} | \beta_i \in \Delta_i\}, 1 \leq i \leq k,$$ 
	where $\Delta_1 = \{1, ..., n_1\}, \Delta_2 = \{n_1+1, ..., n_1+n_2\}, ..., $\\
	$\Delta_k = \{n_1+...+n_{k-1}+1, ..., n_1+...+n_k\}$ and $dim(\mathcal{L}_i) = n_i~ \forall 1 \leq i \leq k.$\\
From \eqref{w9}, \eqref{w8} and \eqref{theta1}, we obtain the scalar curvature of the subspace $\mathcal{L}_i$ as
\begin{align}\label{wtchen1}
\tau(\mathcal{L}_i) =& \frac{n_i(n_i-1)(c_1+c_2)}{4\alpha^2}(p^2+2q) + \frac{(c_1+c_2)}{2\alpha^2}\big[(n_i-1)(n_i~ tr^2\phi - p~ tr\phi) -  \vartheta_1 -  \nonumber \\
&\vartheta_2\big] + \frac{(n_i-1)(c_1-c_2)}{4\alpha}(2~ tr\phi - n_ip) + \sum_{r = n+1}^{m}\sum_{\substack{\gamma_i < \beta_i \\ \gamma_i, \beta_i \in \Delta_i}}[h_{\gamma_i\gamma_i}^rh_{\beta_i\beta_i}^r - (h_{\gamma_i\beta_i}^r)^2],	
\end{align}
where $\vartheta_1 = \cos^2\theta_1 (p~ tr TP_1 +d_1q)$ and $\vartheta_2 = \cos^2\theta_2(p~ tr TP_2 + d_2q).$ \\
Also the scalar curvature of $N$ at $x$ is given by
\begin{align}\label{wtchen2}
2 \tau(x) =& \frac{n(n-1)(c_1+c_2)}{2\alpha^2}(p^2+2q) + \frac{(c_1+c_2)}{\alpha^2}\big[(n-1)(n~ tr^2\phi - p~ tr\phi) - \vartheta_1  \nonumber \\
 &- \vartheta_2\big] + \frac{(n-1)(c_1-c_2)}{2\alpha}(2~ tr\phi - np) + n^2 \parallel \mathcal{H}\parallel^2 - \parallel h\parallel^2.
\end{align}
Let us substitute 
\begin{align}\label{wtchen3}
\xi =&~ 2 \tau(x) - 2 c(n_1, ..., n_k) \parallel\mathcal{H}\parallel^2 - \frac{n(n-1)(c_1+c_2)}{2\alpha^2}(p^2+2q) - \frac{(c_1+c_2)}{\alpha^2} \nonumber \\
&\big[(n-1)(n~ tr^2\phi - p~ tr\phi) - \vartheta_1 - \vartheta_2\big] - \frac{(n-1)(c_1-c_2)}{2\alpha}(2tr\phi - np).
\end{align}
From \eqref{wtchen2} and \eqref{wtchen3}, we obtain
\begin{align}\label{wtchen4}
n^2 \parallel\mathcal{H}\parallel^2 = \mu (\xi + \parallel h\parallel^2),
\end{align}	
where $\mu = (n+k-\sum_{i=1}^{k} n_i).$\\
By selecting the normal vector $e_{n+1}$ to align with the mean curvature vector at $x \in N$ and applying \eqref{wtchen4}, we obtain the following 
\begin{align}\label{wtchen5}
(\sum_{i=1}^{n}z_i)^2 = \mu \big[\xi + \sum_{1 \leq i \neq j \leq n}(h_{ij}^{n+1})^2 + \sum_{i=1}^{n}(z_i)^2 + \sum_{r=n+2}^{m}\sum_{i, j=1}^{n}(h_{ij}^r)^2 \big],
\end{align}
where $z_i = h_{ii}^{n+1},~ 1 \leq i \leq n.$\\
Consider the following framework
\begin{align}
&\bar{z}_1 = z_1, \bar{z}_2 = z_2 + ... + z_{n_1}, \bar{z}_3 = z_{n_1+1} + ... + z_{n_1+n_2}, ..., \nonumber \\
&\bar{z}_{k+1} = z_{n_1+...+n_{k-1}+1} + ... + z_{n_1+...+n_{k}},  \bar{z}_{k+2} = z_{n_1+...+n_{k}+1}, ...,  \nonumber \\
&\bar{z}_{\mu+1} = z_{n}. \nonumber 
\end{align}
Using the above formulation, \eqref{wtchen5} can be rewritten as
\begin{align}\label{wtchen6}
\big(\sum_{i=1}^{\mu + 1}\bar{z}_i\big)^2 =& \mu \Big[\xi + \sum_{i=1}^{\mu + 1}(\bar{z}_i)^2 + \sum_{1 \leq i \neq j \leq n}(h_{ij}^{n+1})^2 + \sum_{r=n+2}^{m}\sum_{i, j=1}^{n}(h_{ij}^r)^2 - \nonumber \\
&\sum_{\substack{2\leq \gamma_1 \neq \beta_1 \leq n_1 \\ \gamma_1, \beta_1 \in \Delta_1}}z_{\gamma_1}z_{\beta_1} - \sum_{i=2}^{k}\sum_{\substack{\gamma_i \neq \beta_i\\ \gamma_i, \beta_i \in \Delta_i}}z_{\gamma_i}z_{\beta_i}\Big] 
\end{align}
Applying the lemma \ref{lemmadelta} to \eqref{wtchen6} yields the following result
\begin{align*}
2 \bar{z}_1\bar{z}_2 \geq \bar{\xi}, 
\end{align*}
where
\begin{align*}
\bar{\xi} =&~ \xi + \sum_{1 \leq i \neq j \leq n}(h_{ij}^{n+1})^2 + \sum_{r=n+2}^{m}\sum_{i, j=1}^{n}(h_{ij}^r)^2 - &\sum_{\substack{2\leq \gamma_1 \neq \beta_1 \leq n_1 \\ \gamma_1, \beta_1 \in \Delta_1}}z_{\gamma_1}z_{\beta_1} - \sum_{i=2}^{k}\sum_{\substack{\gamma_i \neq \beta_i\\ \gamma_i, \beta_i \in \Delta_i}}z_{\gamma_i}z_{\beta_i}.
\end{align*}
Hence, we obtain
\begin{align}\label{wtchen7}
\sum_{i=1}^{k}\sum_{\substack{\gamma_i < \beta_i\\ \gamma_i, \beta_i \in \Delta_i}}z_{\gamma_i}z_{\beta_i} \geq \frac{\xi}{2} + \sum_{1 \leq i < j \leq n}(h_{ij}^{n+1})^2 + \frac{1}{2} \sum_{r=n+2}^{m}\sum_{i, j=1}^{n}(h_{ij}^r)^2.
\end{align}
Consider $\Delta = \Delta_1 \cup ... \cup \Delta_k, ~  \Delta^2 = (\Delta_1 \times \Delta_1) \cup ... \cup (\Delta_k \times \Delta_k),$\\
$\Delta_{k+1} = \{n_1+...+n_k+1, ..., n\}.$ From \eqref{wtchen1} and \eqref{wtchen7}, we get
\begin{align}\label{wtchen10}
\sum_{i=1}^{k}\tau(\mathcal{L}_i) \geq & \frac{(c_1+c_2)}{4\alpha^2} \sum_{i=1}^{k}n_i(n_i-1)(p^2+2q) + \frac{(c_1+c_2)}{2\alpha^2} \sum_{i=1}^{k}(n_i-1)(n_i~ tr^2\phi - p~ tr\phi)  \nonumber \\
 &- \frac{k(c_1+c_2)}{2\alpha^2}(\vartheta_1 + \vartheta_2) + \frac{(c_1-c_2)}{4\alpha}  \sum_{i=1}^{k}(n_i-1)(2~ tr\phi - n_ip) + \frac{\xi}{2}  \nonumber \\
 &+ \sum_{r=n+1}^{m}\sum_{(\gamma, \beta) \notin \Delta^2}(h_{\gamma \beta}^r)^2 + \frac{1}{2} \sum_{r=n+2}^{m}\sum_{i=1}^{k}(\sum_{\gamma_i \in \Delta_i}h_{\gamma_i\gamma_i}^r)^2 + \frac{1}{2} \sum_{r=n+2}^{m}\sum_{i \in \Delta_{k+1}}(h_{ii}^r)^2,
\end{align}
which implies
\begin{align}\label{wtchen11}
\sum_{i=1}^{k}\tau(\mathcal{L}_i) \geq & \frac{(c_1+c_2)}{4\alpha^2} \sum_{i=1}^{k}n_i(n_i-1)(p^2+2q) + \frac{(c_1+c_2)}{2\alpha^2} \sum_{i=1}^{k}(n_i-1)(n_i~ tr^2\phi - p~ tr\phi)  \nonumber \\
 &- \frac{k(c_1+c_2)}{2\alpha^2}(\vartheta_1 + \vartheta_2) + \frac{(c_1-c_2)}{4\alpha}  \sum_{i=1}^{k}(n_i-1)(2~ tr\phi - n_ip) + \frac{\xi}{2}.  
 \end{align}
 Hence, using \eqref{wtchen3} and \eqref{wtchen11}, we obtain the desired inequality \eqref{wc1}. By means of lemma \ref{lemmadelta} and \eqref{wtchen11}, we observe that the equality in \eqref{wc1} holds if and only if we have
 \begin{itemize}
 \item [(i)] $\sum_{i \in \Delta_i}h_{ii}^{n+1} = \nu,$ $1 \leq i \leq k,$ where $\nu$ is some real function on $N,$
 \item [(ii)] $h_{ij}^r = 0,$ $\forall~ r = n+1, ..., m$ and $(i, j) \notin \Delta^2,$
 \item [(iii)] $\sum_{\gamma_i \in \Delta_i}h_{\gamma_i\gamma_i}^r = 0,$ $\forall r = n+2, ..., m$ and $1 \leq i \leq k,$ 
  \item [(iv)] $h_{ii}^r = 0,~ \forall~i \in \Delta_{k+1}.$
 \end{itemize}
 The above conditions imply that the equality in \eqref{wc1} holds if and only if for some suitable orthonormal basis $\{e_1, ..., e_n\}$ of $T_xN$ and $\{e_{n+1}, ..., e_m\}$ of $(T_xN)^{\perp}$, the shape operator is given by \eqref{wce2}. Therefore, we get the assertion.
\end{proof}

\begin{corollary}\label{cor3.2}
		Let $N$ be an $n$-dimensional submanifold (as given in Table \ref{Table 1}) of a locally metallic product space form $M = (M_1(c_1) \times M_2(c_2), g, \phi$) of dimension $m$. Then, the following optimal relationships for the different types of slant submanifolds (given in Table \ref{Table 3.2}) hold and the equality holds if and only if for some suitable orthonormal basis $\{e_1, ..., e_n\}$ of $T_xN$ and $\{e_{n+1}, ..., e_m\}$ of $(T_xN)^{\perp}$, the shape operator is given by \eqref{wce2}.
		\end{corollary}	
	\begin{table}[H]    
	\begin{center}
		\caption{Optimal inequalities involving $\delta$-invariants for submanifolds of locally metallic product space form}
		\label{Table 3.2}
		\vspace{.06cm}
			\begin{tabular}{|l| p{10.5cm}|}
				\hline
				Type of $N$~~ & Optimal inequalities involving $\delta$-invariants\\
				\hline
				Semi-slant & $\delta(n_1, ..., n_k) \leq \frac{(c_1+c_2)}{2\alpha^2}[b(n_1, ..., n_k)(p^2+2q+2~tr^2\phi) - d(n_1, ..., n_k)p~tr\phi + (k-1)(p~ tr TP_1 +d_1q + \cos^2\theta(p~ tr TP_2 + d_2q))] + \frac{(c_1-c_2)}{4\alpha}[2~tr\phi ~d(n_1, ..., n_k) - p~b(n_1, ..., n_k)] + c(n_1, ..., n_k)||\mathcal{H}||^2$\\
				\hline
				Hemi-slant & $\delta(n_1, ..., n_k) \leq \frac{(c_1+c_2)}{2\alpha^2}[b(n_1, ..., n_k)(p^2+2q+2~tr^2\phi) - d(n_1, ..., n_k)p~tr\phi + (k-1) \cos^2\theta (p~ tr TP_1 +d_1q)] + \frac{(c_1-c_2)}{4\alpha}[2~tr\phi ~d(n_1, ..., n_k) - p~b(n_1, ..., n_k)] + c(n_1, ..., n_k)||\mathcal{H}||^2$\\
				\hline
				Semi-invariant & $\delta(n_1, ..., n_k) \leq \frac{(c_1+c_2)}{2\alpha^2}[b(n_1, ..., n_k)(p^2+2q+2~tr^2\phi) - d(n_1, ..., n_k)p~tr\phi + (k-1)(p~ tr TP_1 +d_1q)] + \frac{(c_1-c_2)}{4\alpha}[2~tr\phi~ d(n_1, ..., n_k) - p~b(n_1, ..., n_k)] + c(n_1, ..., n_k)||\mathcal{H}||^2$\\
				\hline
				Slant & $\delta(n_1, ..., n_k) \leq \frac{(c_1+c_2)}{2\alpha^2}[b(n_1, ..., n_k)(p^2+2q+2~tr^2\phi) - d(n_1, ..., n_k)p~tr\phi + (k-1)\cos^2\theta (p~ tr T + nq)] + \frac{(c_1-c_2)}{4\alpha}[2~tr\phi ~d(n_1, ..., n_k) - p~b(n_1, ..., n_k)] + c(n_1, ..., n_k)||\mathcal{H}||^2$\\
				\hline
			\end{tabular}	
		\end{center}
	\end{table}	
%


	\subsection{Optimal inequalities involving Ricci curvature and shape operator}\label{sec2.3w}
This section comprises the optimal inequalities involving Ricci curvature and shape operator of submanifolds in metallic Riemannian space forms.
	
\begin{theorem}\label{wriccith}
	Let $N$ be an $n$-dimensional bi-slant submanifold of a locally metallic product space form $M = (M_1(c_1) \times M_2(c_2), g, \phi$) of dimension $m$. Then, for an integer $k,$ $2 \leq k \leq n$, we have
	\begin{align}\label{shapeth1}	
	(tr(\mathcal{A}_H))^2 \geq &~ n(n-1)\Omega_k + \omega + \frac{(c_1+c_2)}{\alpha^2}[\cos^2{\theta_1}(p~tr TP_1 + d_1q)   \nonumber \\
&+ \cos^2\theta_2(p~tr TP_2 + d_2q)],
	\end{align}
	where
	\begin{align*}
	& \omega = n\omega_1 + [(n-1)\omega_3 - \omega_2] tr\phi - \frac{n(n-1)(c_1+ c_2)}{\alpha^2} tr^2\phi, \\
	& \omega_1 = \frac{p(n-1)( c_1- c_2)}{2\alpha} - \frac{(n-1)(c_1+c_2)(p^2+2q)}{2\alpha^2},\\
	& \omega_2 = \frac{(n-1)}{2\alpha^2}[p(c_1+c_2) + \alpha(c_1- c_2)],~ \omega_3 = \frac{p(c_1+c_2) - \alpha(c_1- c_2)}{2\alpha^2}
	 \end{align*}
and the equality in \eqref{shapeth1} holds if and only if the Ricci curvature of all the $k$-plane sections of $T_xN$ is constant and $\mathcal{A}_r = 0,~ \forall r \in \{n+2, ..., m\}.$ 
	\end{theorem}
	\begin{proof}
	Choose an orthonormal basis $\{e_1, ..., e_n\}$ of the tangent space $T_xN$ and $\{e_{n+1}, ..., e_m\}$ of the normal space $(T_xN)^{\perp}$ of $N$ at $x \in N$ such that $e_{n+1}$ is parallel to the mean curvature vector and $\{e_1, ..., e_n\}$ diagonalize the shape operator $\mathcal{A}_{n+1}.$ Then, we have
\begin{align}\label{shape1}
\mathcal{A}_{n+1} = \mathrm{diag}(a_1, \dots, a_n), ~\mathcal{A}_{r} = [h_{ij}^r], ~\sum_{i=1}^{n}h_{ii}^r = 0,~\forall r \in \{n+2, ..., m\}.
\end{align}
From \eqref{w6}, \eqref{w7}, \eqref{w9}, we have
\begin{align}\label{shape2}
a_ia_j =& K(e_i \wedge e_j) - \frac{(c_1 + c_2)}{4\alpha^2}\big\{\alpha^2 + 4[g(\phi e_j, e_j) g(\phi e_i, e_i) - [g(\phi e_i, e_j)]^2] + p^2  \nonumber \\
& - 2p[g(\phi e_j, e_j) + g(\phi e_i, e_i)] \big\} - \frac{(c_1 - c_2)}{2\alpha}\big[g(\phi e_j, e_j) + g(\phi e_i, e_i) - p\big]  \nonumber \\
& + \sum_{r=n+2}^{m}[(h_{ij}^r)^2 - h_{ii}^rh_{jj}^r].
\end{align}
Let $\{e_{i_1}, ..., e_{i_k}\}$ denote the basis of the $k$-plane section $\mathcal{L}_{i_1, ..., i_k}$. Then, taking $i=1$ and $j = i_2, ..., i_l$, we obtain
\begin{align}\label{shape3}
a_1\sum_{j=2}^{k}a_{i_j} =& Ric_{\mathcal{L}_{i_1, ..., i_k}}(e_1) - \frac{(c_1 + c_2)}{4}(k-1) - \bigg[\frac{(c_1 + c_2)}{\alpha^2}g(\phi e_1, e_1) + \frac{p(c_1 + c_2)}{2\alpha^2}   \nonumber \\
&- \frac{(c_1 - c_2)}{2\alpha} \bigg]\sum_{j=2}^{k} g(\phi e_{i_j}, e_{i_j}) + \frac{(c_1 + c_2)}{\alpha^2}\sum_{j=2}^{k} g(\phi e_1, e_{i_j})^2 + \frac{(k-1)}{2\alpha}    \nonumber \\
&\bigg[\frac{p(c_1 + c_2)}{\alpha} - c_1 + c_2\bigg]g(\phi e_1, e_1) - \frac{(k-1)p}{2\alpha}\bigg[\frac{p(c_1 + c_2)}{2\alpha} - c_1 + c_2\bigg]  \nonumber \\
&+ \sum_{r=n+2}^{m}\sum_{j=2}^{k} [(h_{1{i_j}}^r)^2 - h_{11}^rh_{{i_j}{i_j}}^r].
\end{align}	
Summing up over all such $k$-plane sections, we get
\begin{align}\label{shape4}
a_1\sum_{i=1}^{n}a_i =& \frac{1}{{}^{n-2}C_{k-2}}\sum_{2 \leq i_2 < ... < i_k \leq n} Ric_{\mathcal{L}_{i_1, ..., i_k}}(e_1) + \omega_1 - \omega_2g(\phi e_1, e_1) + \bigg[\omega_3    \nonumber \\
&- \frac{(c_1 + c_2)}{\alpha^2}g(\phi e_1, e_1)\bigg] \sum_{i=2}^{n} g(\phi e_{i}, e_{i}) + \frac{(c_1 + c_2)}{\alpha^2}\sum_{i=2}^{n} g(\phi e_1, e_{i})^2    \nonumber \\
&+ \sum_{r=n+2}^{m}\sum_{i=1}^{n} [(h_{1i}^r)^2. 
\end{align}
Using definition \ref{omegadef}, we obtain
\begin{align}\label{shape5}
a_1\sum_{i=1}^{n}a_i \geq &(n-1)\Omega_k + \omega_1 - \omega_2g(\phi e_1, e_1) + \bigg[\omega_3 - \frac{(c_1 + c_2)}{\alpha^2}g(\phi e_1, e_1)\bigg]   \nonumber \\
&\sum_{i=2}^{n}  g(\phi e_{i}, e_{i}) + \frac{(c_1 + c_2)}{\alpha^2}\sum_{i=2}^{n} g(\phi e_1, e_{i})^2.    
\end{align}
Similar result holds if $i \in \{2, ..., n\}$ in \eqref{shape2}. Hence, we obtain
\begin{align}\label{shape6}
a_j\sum_{i=1}^{n}a_i \geq &(n-1)\Omega_k + \omega_1 - \omega_2g(\phi e_j, e_j) + \bigg[\omega_3 - \frac{(c_1 + c_2)}{\alpha^2}g(\phi e_j, e_j)\bigg]   \nonumber \\
&\sum_{\substack{i=2 \\ i \neq j}}^{n}  g(\phi e_{i}, e_{i}) + \frac{(c_1 + c_2)}{\alpha^2}\sum_{\substack{i=2 \\ i \neq j}}^{n} g(\phi e_j, e_{i})^2,~ 1\leq j \leq n.  
\end{align}
Taking summation over $1 \leq j \leq n$ in \eqref{shape6}, we get
\begin{align}\label{shape8}
(tr(\mathcal{A}_H))^2 \geq &~ n(n-1)\Omega_k + \omega + \frac{(c_1+c_2)}{\alpha^2}[\cos^2{\theta_1}(p~tr TP_1 + d_1q)   \nonumber \\
&+ \cos^2\theta_2(p~tr TP_2 + d_2q)],
\end{align}
which proves the desired inequality \eqref{shapeth1}. This gives the relationship for the lower bound of eigen values of the shape operator $\mathcal{A}_H$. Also from \eqref{shape1}, \eqref{shape5} and \eqref{shape8}, we conclude that the equality in \eqref{shapeth1} holds if and only if the Ricci curvature of all the $k$-plane sections of $T_xN$ vanishes and $\mathcal{A}_r = 0,~ \forall r \in \{n+2, ..., m\}.$
	\end{proof}

\begin{corollary}\label{cor3.33}
		Let $N$ be an $n$-dimensional submanifold (as given in Table \ref{Table 1}) of a locally metallic product space form $M = (M_1(c_1) \times M_2(c_2), g, \phi$) of dimension $m$. Then, the following optimal relationships involving shape operator and Ricci curvature for the different types of slant submanifolds (given in Table \ref{Table 3.5}) hold and the equality holds if and only if the Ricci curvature of all the $k$-plane sections of $T_xN$ vanishes and $\mathcal{A}_r = 0,~ \forall r \in \{n+2, ..., m\}.$ 
\end{corollary}

\begin{table}[H]    
	\begin{center}
		\caption{Optimal inequalities involving Ricci curvature and shape operator for submanifolds of locally metallic product space form}
		\label{Table 3.5}
		\vspace{.06cm}
			\begin{tabular}{|l| p{10.5cm}|}
				\hline
				Type of $N$~~ & Optimal inequalities involving Ricci curvature and shape operator\\
				\hline
				Semi-slant & $(tr(\mathcal{A}_H))^2 \geq n(n-1)\Omega_k + \omega + \frac{(c_1+c_2)}{\alpha^2}[p~tr TP_1 + d_1q + \cos^2\theta(p~tr TP_2 + d_2q)]$\\
				\hline
				Hemi-slant & $(tr(\mathcal{A}_H))^2 \geq n(n-1)\Omega_k + \omega + \frac{(c_1+c_2)}{\alpha^2}\cos^2{\theta}(p~tr TP_1 + d_1q)$\\
				\hline
				Semi-invariant & $(tr(\mathcal{A}_H))^2 \geq n(n-1)\Omega_k + \omega + \frac{(c_1+c_2)}{\alpha^2}(p~tr TP_1 + d_1q)$\\
				\hline
				Slant & $(tr(\mathcal{A}_H))^2 \geq n(n-1)\Omega_k + \omega + \frac{(c_1+c_2)}{\alpha^2}\cos^2{\theta}(p~tr T + nq)$\\
				\hline
			\end{tabular}	
		\end{center}
	\end{table}	

\begin{theorem}\label{wriccilemma}
Let $N$ be an $n$-dimensional bi-slant submanifold of a locally metallic product space form $M = (M_1(c_1) \times M_2(c_2), g, \phi$) of dimension $m$. Then, we have
\begin{align}\label{wricci1}
n(n-1) \parallel H \parallel^2\geq&~ 2\tau - \frac{n(n-1)(c_1+c_2)}{2\alpha^2}(p^2+2q) + \frac{(c_1+c_2)}{\alpha^2}\big[\cos^2\theta_1(p~ tr TP_1  \nonumber \\ 
&+ d_1q) + \cos^2\theta_2(p~ tr TP_2 + d_2q) - (n-1)(n~ tr^2\phi - p~ tr\phi) \big]   \nonumber \\ 
&- \frac{(n-1)(c_1-c_2)}{2\alpha}(2~ tr\phi - np)
\end{align}	
and the equality holds at a point $x \in N$ if and only if $x$ is a totally umbilical point.
\end{theorem}
\begin{proof}
Choose an orthonormal basis $\{e_1, ..., e_n\}$ of the tangent space $T_xN$ and $\{e_{n+1}, ..., e_m\}$ of the normal space $(T_xN)^{\perp}$ of $N$ at $x \in N$ such that $e_{n+1}$ is parallel to the mean curvature vector and $\{e_1, ..., e_n\}$ diagonalize the shape operator $\mathcal{A}_{n+1}.$ Using \eqref{w9}, \eqref{w8} and \eqref{theta1}, we obtain 
\begin{align}\label{wricci3}
n^2 \parallel H \parallel^2~=~& 2\tau + \sum_{i=1}^{n}(a_i)^2 + \sum_{i, j=1}^{n}\sum_{r=n+2}^{m}(h_{ij}^{r})^2 - \frac{n(n-1)(c_1+c_2)}{2\alpha^2}(p^2+2q)  \nonumber \\ 
&+ \frac{(c_1+c_2)}{\alpha^2}\big[\cos^2\theta_1(p~ tr TP_1 + d_1q) + \cos^2\theta_2(p~ tr TP_2 + d_2q)  \nonumber \\ 
&- (n-1)(n~ tr^2\phi - p~ tr\phi) \big] - \frac{(n-1)(c_1-c_2)}{2\alpha}(2~ tr\phi - np),
\end{align}
where $a_1, ..., a_n$ are eigen values of $\mathcal{A}_{n+1}.$
Using Cauchy-Schwarz inequality, we have
\begin{align}\label{wricci114}
\sum_{i=1}^{n}(a_i)^2 \geq n \parallel H \parallel^2.
\end{align}
Using \eqref{wricci114} in \eqref{wricci3}, we obtain the desired optimal relation \eqref{wricci1}. The equality holds at a point $x \in N$ if and only if the following two conditions hold
\begin{itemize}
\item[(i)] $a_i = a_j$ $\forall$ $i, j = 1, ..., n,$
\item[(ii)] $h_{ij}^{r} = 0$$\forall$ $i, j = 1, ..., n, r \in \{n+2, ..., m\}.$
\end{itemize}
Therefore, we conclude that the equality in \eqref{wricci1} holds at a point $x \in N$ if and only if $x$ is a totally umbilical point. Hence, we get the assertion.
\end{proof}

\begin{corollary}\label{cor3.34}
		Let $N$ be an $n$-dimensional submanifold (as given in Table \ref{Table 1}) of a locally metallic product space form $M = (M_1(c_1) \times M_2(c_2), g, \phi$) of dimension $m$. Then, the following optimal relationships for the different types of slant submanifolds (given in Table \ref{Table 3.6}) hold and the equality holds at a point $x \in N$ if and only if $x$ is a totally umbilical point.
\end{corollary}
\begin{table}[H]   
	\begin{center}
		\caption{Optimal inequalities involving squared mean curvature and scalar curvature for submanifolds of locally metallic product space form}
		\label{Table 3.6}
		\vspace{.06cm}
			\begin{tabular}{|l| p{10.4cm}|}
				\hline
				Type of $N$~~ & Optimal inequalities involving squared mean curvature and scalar curvature\\
				\hline
				Semi-slant & $n(n-1) ||H||^2\geq 2\tau - \frac{n(n-1)(c_1+c_2)}{2\alpha^2}(p^2+2q) + \frac{(c_1+c_2)}{\alpha^2}\big[p~tr TP_1 + d_1q + \cos^2\theta(p~tr TP_2 + d_2q) - (n-1)(n~ tr^2\phi - p~ tr\phi) \big] - \frac{(n-1)(c_1-c_2)}{2\alpha}(2~ tr\phi - np)$\\
				\hline
				Hemi-slant & $n(n-1) ||H||^2\geq 2\tau - \frac{n(n-1)(c_1+c_2)}{2\alpha^2}(p^2+2q) + \frac{(c_1+c_2)}{\alpha^2}\big[\cos^2{\theta}(p~tr TP_1 + d_1q) - (n-1)(n~ tr^2\phi - p~ tr\phi) \big] - \frac{(n-1)(c_1-c_2)}{2\alpha}(2~ tr\phi - np)$\\
				\hline
				Semi-invariant & $n(n-1) ||H||^2\geq 2\tau - \frac{n(n-1)(c_1+c_2)}{2\alpha^2}(p^2+2q) + \frac{(c_1+c_2)}{\alpha^2}[(p~tr TP_1 + d_1q) - (n-1)(n~ tr^2\phi - p~ tr\phi)] - \frac{(n-1)(c_1-c_2)}{2\alpha}(2~ tr\phi - np)$\\
				\hline
				Slant & $n(n-1) ||H||^2\geq 2\tau - \frac{n(n-1)(c_1+c_2)}{2\alpha^2}(p^2+2q) + \frac{(c_1+c_2)}{\alpha^2}[\cos^2{\theta}(p~tr T + nq) - (n-1)(n~ tr^2\phi - p~ tr\phi)] - \frac{(n-1)(c_1-c_2)}{2\alpha}(2~ tr\phi - np)$\\
				\hline
			\end{tabular}	
		\end{center}
	\end{table}	

	
	\subsection{Optimal inequalities involving generalized normalized $\delta$-Casorati curvatures}\label{sec2.4w}
In this section, we derive optimal inequalities involving generalized normalized $\delta$-Casorati curvatures for submanifolds of metallic Riemannian spaceform.
\begin{theorem}\label{casorati1}	
Let $N$ be an $n$-dimensional bi-slant submanifold of a locally metallic product space form $M = (M_1(c_1) \times M_2(c_2), g, \phi$) of dimension $m$. Then 
\begin{itemize}
\item [(i)] for any real number $u$ such that $0 < u < n(n-1),$ the generalized normalized $\delta$-Casorati curvature $\delta_{\mathcal{C}}(u; n-1)$ satisfies
\begin{align}\label{cas1}
2\tau \leq & \delta_{\mathcal{C}}(u; n-1) + \frac{n(n-1)(c_1+c_2)}{2\alpha^2}(p^2+2q) + \frac{(c_1+c_2)}{\alpha^2}\big[(n-1)  \nonumber   \\ 
&(n~ tr^2\phi - p~ tr\phi) - \cos^2\theta_1(p~ tr TP_1 +d_1q) - \cos^2\theta_2(p~ tr TP_2 + d_2q)\big]   \nonumber   \\ 
& + \frac{(n-1)(c_1-c_2)}{2\alpha}(2~ tr\phi - np),
\end{align}
\item [(ii)] for any real number $u$ such that $u > n(n-1),$ the generalized normalized $\delta$-Casorati curvature $\hat{\delta}_{\mathcal{C}}(u; n-1)$ satisfies
\begin{align}\label{cas2}
2\tau \leq &  \hat{\delta}_{\mathcal{C}}(u; n-1) + \frac{n(n-1)(c_1+c_2)}{2\alpha^2}(p^2+2q) + \frac{(c_1+c_2)}{\alpha^2}\big[(n-1)  \nonumber   \\ 
&(n~ tr^2\phi - p~ tr\phi) - \cos^2\theta_1(p~ tr TP_1 +d_1q) - \cos^2\theta_2(p~ tr TP_2 + d_2q)\big]   \nonumber   \\ 
& + \frac{(n-1)(c_1-c_2)}{2\alpha}(2~ tr\phi - np)
\end{align}
\end{itemize}
and the equality holds in \eqref{cas1} and \eqref{cas2} at $x \in N$ if and only if the submanifold $N$ is invariantly quasi-umbilical with trivial normal connection in $M$ and for some suitable orthonormal basis $\{e_1, ..., e_n\}$ of $T_xN$ and $\{e_{n+1}, ..., e_m\}$ of $(T_xN)^{\perp}$, the shape operator takes the following form
\begin{equation}\label{cas3}
			\mathcal{A}_{n+1} =
			\begin{pmatrix}  
				a & \cdots & 0         & 0\\
			\vdots  & \ddots & \vdots & \vdots \\
				0 & \cdots & a         & 0\\
				0 & \cdots & 0         &(\frac{n^2-n}{u})a  \\
			\end{pmatrix}, 
			\quad
		\mathcal{A}_{n+2} = ... = \mathcal{A}_{m} = 0.
		\end{equation}

\end{theorem}
\begin{proof}
Consider an orthonormal basis $\{e_1, ..., e_n\}$ of $T_xN$ and $\{e_{n+1}, ..., e_m\}$ of $(T_xN)^{\perp}$ at $x \in N$.
Using  \eqref{w9}, \eqref{w8} and \eqref{theta1}, we obtain the scalar curvature of $N$ as 
\begin{align}\label{cas4}
2 \tau &= \frac{n(n-1)(c_1+c_2)}{2\alpha^2}(p^2+2q) + \frac{(c_1+c_2)}{\alpha^2}\big[(n-1)(n~ tr^2\phi - p~ tr\phi) - \cos^2\theta_1 \nonumber   \\ 
&(p~ tr TP_1 +d_1q) - \cos^2\theta_2(p~ tr TP_2 + d_2q)\big] + \frac{(n-1)(c_1-c_2)}{2\alpha}(2~ tr\phi - np)   \nonumber  \\ 
&+ n^2 \parallel\mathcal{H}\parallel^2 - n\mathcal{C}.	
\end{align}
Consider a hyperplane $\mathcal{W}$ of $T_xN.$ Without loss of generality, we can assume that $\{e_1, ..., e_{n-1}\}$ span $\mathcal{W}.$ Consider a function $\mathcal{V}$ given by a quadratic polynomial in the components of the second fundamental form as
\begin{align}\label{cas6}
\mathcal{V} =&~ u \mathcal{C} + \frac{(n+u)(n^2-n-u)}{nu}\sum_{r=n+1}^{m}\sum_{1 \leq i, j \leq n-1}(h_{ij}^r)^2 - 2\tau + \frac{n(n-1)(c_1+c_2)}{2\alpha^2}  \nonumber   \\ 
&(p^2+2q) + \frac{(c_1+c_2)}{\alpha^2}\big[(n-1)(n~ tr^2\phi - p~ tr\phi) - \cos^2\theta_1(p~ tr TP_1 +d_1q)    \nonumber   \\ 
&- \cos^2\theta_2(p~ tr TP_2 + d_2q)\big] + \frac{(n-1)(c_1-c_2)}{2\alpha}(2~ tr\phi - np). 
\end{align}
From \eqref{cas4} and \eqref{cas6}, we obtain
\begin{align}\label{cas8}
\mathcal{V} =& \frac{n^2+n(u-1)-2u}{u}\sum_{r=n+1}^{m}\sum_{i=1}^{n-1}(h_{ii}^r)^2 + \frac{2(n+u)(n-1)}{u}\sum_{r=n+1}^{m}\sum_{1 \leq i < j \leq n-1}(h_{ij}^r)^2   \nonumber \\ 
&+ \frac{2(n+u)}{n}\sum_{r=n+1}^{m}\sum_{i=1}^{n-1}(h_{in}^r)^2 + \frac{u}{n}\sum_{r=n+1}^{m}(h_{nn}^r)^2 - 2\sum_{r=n+1}^{m}\sum_{1 \leq i < j \leq n}h_{ii}^rh_{jj}^r.
\end{align}
The solutions of the following system of linear homogeneous equations give the critical points $h^c = (h_{11}^{n+1}, h_{12}^{n+1}, ..., h_{nn}^{n+1}, ..., h_{11}^{m}, ..., h_{nn}^{m})$ of $\mathcal{V}$:
\begin{equation}\label{cas9}
\begin{cases}
\dfrac{\partial\mathcal{V}}{\partial h_{ii}^r} = \dfrac{2(n+u)(n-1)}{u}h_{ii}^r - 2\sum_{k=1}^{n}h_{kk}^r = 0,       \\
\dfrac{\partial\mathcal{V}}{\partial h_{nn}^r} = \dfrac{2u}{n}h_{nn}^r - 2\sum_{k=1}^{n-1}h_{kk}^r = 0,        \\ 
\dfrac{\partial\mathcal{V}}{\partial h_{ij}^r} = \dfrac{4(n+u)(n-1)}{u}h_{ij}^r = 0,      \\
\dfrac{\partial\mathcal{V}}{\partial h_{in}^r} = \dfrac{4(n+u)}{n}h_{in}^r,      
\end{cases}
 \end{equation}    
 where $i, j \in \{1, ..., n\}, i \neq j, r \in \{n+1, ..., m\}.$  
 From above system of equations, we observe that every solution $h^c$ of \eqref{cas9} has $h_{ij}^r = 0$ for $i \neq j$, which characterize submanifolds with a trivial normal connection. Also the determinant of the first two equations in \eqref{cas9} is zero, indicating the existence of solutions that do not correspond to totally geodesic submanifolds. Moreover, the eigenvalues of the Hessian matrix $\mathcal{H}(\mathcal{V})$ of $\mathcal{V}$ are given by   
 \begin{align}
&\lambda_{11}^r = 0,~ \lambda_{22}^r = \frac{n^2(n-1) + u^2}{nu},~ \lambda_{33}^r = ... = \lambda_{33}^r = \frac{2(n-1)(n+u)}{u}, \nonumber \\
&\lambda_{ij}^r = \frac{4(n-1)(n+u)}{u},~  \lambda_{in}^r = \frac{4(n+u)}{n}, \nonumber 
 \end{align}
 for $i, j \in \{1, ..., n-1\}, i \neq j.$\\
 Therefore, we conclude that $\mathcal{V}$ is a parabolic function that attains a minimum at $h^c,$ satisfying $\mathcal{V}(h^c) = 0$ for the solution $h^c$ of the system \eqref{cas9}. Hence $\mathcal{V} \geq 0.$\\
 Thus, we obtain from \eqref{cas6},
\begin{align}\label{caso1}
2\tau \leq &~ u \mathcal{C} + \frac{(n-1)(n+u)(n^2-n-u)}{nu}\mathcal{C}(\mathcal{W}) + \frac{n(n-1)(c_1+c_2)}{2\alpha^2}(p^2+2q)  \nonumber   \\ 
&+ \frac{(c_1+c_2)}{\alpha^2}\big[(n-1)(n~ tr^2\phi - p~ tr\phi) - \cos^2\theta_1(p~ tr TP_1 +d_1q)    \nonumber   \\ 
&- \cos^2\theta_2(p~ tr TP_2 + d_2q)\big] + \frac{(n-1)(c_1-c_2)}{2\alpha}(2~ tr\phi - np).
\end{align}
Now as \eqref{caso1} holds for all such hyperplanes $\mathcal{W}$ of $T_xN,$ hence we obtain the desired inequalities \eqref{cas1} and \eqref{cas2}. Also the equality holds in \eqref{cas1} and \eqref{cas2} if and only if we have the following
\begin{itemize}
\item[(i)] $h_{ij}^r = 0, 1 \leq i \neq j \leq n,$
\item[(ii)] $h_{ii}^r = h_{jj}^r = \frac{u}{n(n-1)}h_{nn}^r, 1 \leq i, j \leq n-1, n+1 \leq r \leq m.$
\end{itemize}
Hence, from the above conditions, it is clear that the equality holds in \eqref{cas1} and \eqref{cas2} at $x \in N$ if and only if the submanifold $N$ is invariantly quasi-umbilical with trivial normal connection in $M$ and for some suitable orthonormal basis $\{e_1, ..., e_n\}$ of $T_xN$ and $\{e_{n+1}, ..., e_m\}$ of $(T_xN)^{\perp}$, the shape operator takes the form as given in \eqref{cas3}.
 \end{proof}

	\begin{corollary}\label{cor3.4}
		Let $N$ be an $n$-dimensional submanifold (as given in Table \ref{Table 1}) of a locally metallic product space form $M = (M_1(c_1) \times M_2(c_2), g, \phi$) of dimension $m$. Then, the following optimal relationships involving generalized normalized $\delta$-Casorati curvatures for the different types of slant submanifolds (given in Table \ref{Table 3.4}) hold and the equality holds if and only if the submanifold $N$ is invariantly quasi-umbilical with trivial normal connection in $M$ and for some suitable orthonormal basis $\{e_1, ..., e_n\}$ of $T_xN$ and $\{e_{n+1}, ..., e_m\}$ of $(T_xN)^{\perp}$, the shape operator takes the form as given in \eqref{cas3}.
\end{corollary}	
	\begin{table} [H]  
	\begin{center}
		\caption{Optimal inequalities involving generalized normalized $\delta$-Casorati curvatures for submanifolds of locally metallic product space form}
		\label{Table 3.4}
		\vspace{.06cm}
			\begin{tabular}{|l| p{10.5cm}|}
				\hline
				Type of $N$~~ & Optimal inequalities involving generalized normalized $\delta$-Casorati curvatures\\
				\hline
				Semi-slant & $(i)$ $2\tau \leq \delta_{\mathcal{C}}(u; n-1) + \frac{n(n-1)(c_1+c_2)}{2\alpha^2}(p^2+2q) + \frac{(c_1+c_2)}{\alpha^2}[(n-1)(n~ tr^2\phi - p~ tr\phi) - p~ tr TP_1 - d_1q) - \cos^2\theta(p~ tr TP_2 + d_2q)] + \frac{(n-1)(c_1-c_2)}{2\alpha}(2~ tr\phi - np)$ \vspace{.1cm} \\ 
			&$(ii)$ $2\tau \leq \hat{\delta}_{\mathcal{C}}(u; n-1) + \frac{n(n-1)(c_1+c_2)}{2\alpha^2}(p^2+2q) + \frac{(c_1+c_2)}{\alpha^2}[(n-1)(n~ tr^2\phi - p~ tr\phi) - p~ tr TP_1 - d_1q - \cos^2\theta(p~ tr TP_2 + d_2q)] + \frac{(n-1)(c_1-c_2)}{2\alpha}(2~ tr\phi - np)$ \\
				\hline
				Hemi-slant & $(i)$ $2\tau \leq \delta_{\mathcal{C}}(u; n-1) + \frac{n(n-1)(c_1+c_2)}{2\alpha^2}(p^2+2q) + \frac{(c_1+c_2)}{\alpha^2}[(n-1)(n~ tr^2\phi - p~ tr\phi) - \cos^2\theta(p~ tr TP_1 +d_1q)] + \frac{(n-1)(c_1-c_2)}{2\alpha}(2~ tr\phi - np)$ \vspace{.1cm} \\ 
			&$(ii)$ $2\tau \leq \hat{\delta}_{\mathcal{C}}(u; n-1) + \frac{n(n-1)(c_1+c_2)}{2\alpha^2}(p^2+2q) + \frac{(c_1+c_2)}{\alpha^2}[(n-1)(n~ tr^2\phi - p~ tr\phi) - \cos^2\theta(p~ tr TP_1 +d_1q)] + \frac{(n-1)(c_1-c_2)}{2\alpha}(2~ tr\phi - np)$ \\
				\hline
				Semi-invariant & $(i)$ $2\tau \leq \delta_{\mathcal{C}}(u; n-1) + \frac{n(n-1)(c_1+c_2)}{2\alpha^2}(p^2+2q) + \frac{(c_1+c_2)}{\alpha^2}[(n-1)(n~ tr^2\phi - p~ tr\phi) - p~ tr TP_1 - d_1q] + \frac{(n-1)(c_1-c_2)}{2\alpha}(2~ tr\phi - np)$ \vspace{.1cm} \\ 
			&$(ii)$ $2\tau \leq \hat{\delta}_{\mathcal{C}}(u; n-1) + \frac{n(n-1)(c_1+c_2)}{2\alpha^2}(p^2+2q) + \frac{(c_1+c_2)}{\alpha^2}[(n-1)(n~ tr^2\phi - p~ tr\phi) - p~ tr TP_1 - d_1q] + \frac{(n-1)(c_1-c_2)}{2\alpha}(2~ tr\phi - np)$ \\
				\hline
				Slant & $(i)$ $2\tau \leq \delta_{\mathcal{C}}(u; n-1) + \frac{n(n-1)(c_1+c_2)}{2\alpha^2}(p^2+2q) + \frac{(c_1+c_2)}{\alpha^2}[(n-1)(n~ tr^2\phi - p~ tr\phi) - \cos^2\theta(p~ tr T + nq)] + \frac{(n-1)(c_1-c_2)}{2\alpha}(2~ tr\phi - np)$ \vspace{.1cm} \\ 
			&$(ii)$ $2\tau \leq \hat{\delta}_{\mathcal{C}}(u; n-1) + \frac{n(n-1)(c_1+c_2)}{2\alpha^2}(p^2+2q) + \frac{(c_1+c_2)}{\alpha^2}[(n-1)(n~ tr^2\phi - p~ tr\phi) - \cos^2\theta(p~ tr T +nq)] + \frac{(n-1)(c_1-c_2)}{2\alpha}(2~ tr\phi - np)$ \\
				\hline
			\end{tabular}	
		\end{center}
	\end{table}	


		\section*{Conclusion} 
In this study, we have established optimal relationships between key intrinsic and extrinsic curvature invariants for various classes of submanifolds within the context of metallic Riemannian product space forms. For an isometric immersion, the squared mean curvature at each point represents the amount of tension exerted by the ambient space on the submanifold. The derived optimal inequalities involving $\delta$-invariant (an intrinsic invariant) and squared mean curvature (an extrinsic invariant) characterize the conditions for an ideal immersion, where the submanifold experiences the least possible tension from the surrounding space. Additionally, the optimal inequalities involving $\delta$-Casorati curvature (an extrinsic invariant) provide meaningful characterizations of the geometric behaviour of these submanifolds. These results provide a strong foundation for further research. They can be extended to other types of submanifolds by using warped product structures, which allow for more general geometric settings. Additionally, by considering soliton structures such as Ricci solitons, Yamabe solitons, and Ricci–Bourguignon solitons, we can develop solitonic inequalities that offer a better understanding of how curvature and geometric flows interact on these submanifolds.

\section*{Acknowledgments}
The first author is thankful to UGC for providing financial assistance in terms of SRF scholarship vide NTA Ref. No.: $201610070797$(CSIR-UGC NET June 2020). The second author is thankful to the Department of Science and Technology (DST) Government of India for providing financial assistance in terms of FIST project (TPN-69301) vide the letter with Ref. No.: (SR/FST/MS-1/2021/104).

	\noindent 	H. Kaur and G. Shanker\newline
	Department of Mathematics and Statistics\newline
	Central University of Punjab\newline
	Bathinda, Punjab-151401, India.\newline
	Email: harmandeepkaur1559@gmail.com, \newline
	gauree.shanker@cup.edu.in.
\end{document}